\documentclass{gtpart}     
\title{Vanishing of $L^{2}$-Betti numbers and failure of acylindrical
hyperbolicity of matrix groups over rings}

\author{Feng Ji}
\givenname{Feng}
\surname{Ji}
\address{Infinitus, Nanyang Technological University, 50 Nanyang Ave, S2-B4b-05,
Singapore 639798.}
\email{jifeng@ntu.edu.sg}
\urladdr{}

\author{Shengkui Ye}
\givenname{Shengkui}
\surname{Ye}
\address{Department of Mathematical Sciences, Xi'an Jiaotong-Liverpool University,
111 Ren Ai Road, Suzhou, Jiangsu, China 215123.}
\email{Shengkui.Ye@xjtlu.edu.cn}
\urladdr{https://yeshengkui.wordpress.com/}


\arxivreference{}
\arxivpassword{}

\newtheorem{theorem}{Theorem}[section]

\newtheorem{corollary}[theorem]{Corollary}

\newtheorem{definition}[theorem]{Definition}
\newtheorem{example}[theorem]{Example}

\newtheorem{lemma}[theorem]{Lemma}

\newtheorem{proposition}[theorem]{Proposition}
\newtheorem{remark}[theorem]{Remark}

\begin{document}

\begin{abstract}
Let $R$ be an infinite commutative ring with identity and $n\geq 2$ be an
integer. We prove that for each integer $i=0,1,\cdots ,n-2,$ the $L^{2}$%
-Betti number $b_{i}^{(2)}(G)=0,$ $\ $when $G=\mathrm{GL}_{n}(R)$ the
general linear group, $\mathrm{SL}_{n}(R)$ the special linear group, $%
E_{n}(R)$ the group generated by elementary matrices. When $R$ is an
infinite principal ideal domain, similar results are obtained for $\mathrm{Sp%
}_{2n}(R)$ the symplectic group, $\mathrm{ESp}_{2n}(R)$ the elementary
symplectic group, $\mathrm{O}(n,n)(R)$ the split orthogonal group or $%
\mathrm{EO}(n,n)(R)$ the elementary orthogonal group. Furthermore, we prove
that $G$ is not acylindrically hyperbolic if $n\geq 4$. We also prove
similar results for a class of noncommutative rings. The proofs are based on
a notion of $n$-rigid rings.
\end{abstract}

\maketitle

\section{Introduction}

In this article, we study the $s$-normality of subgroups of matrix groups
over rings together with two applications. Firstly, the low-dimensional $%
L^{2}$-Betti numbers of matrix groups are proved to be zero. Secondly, the
matrix groups are proved to be not acylindrically hyperbolic in the sense of
Dahmani--Guirardel--Osin \cite{dgo} and Osin \cite{Os}. Let us briefly
review the relevant background.

Let $G$ be a discrete group. Denote by
\begin{equation*}
l^{2}(G)=\{f:G\rightarrow \mathbb{C}\mid \sum_{g\in G}\Vert f(g)\Vert
^{2}<+\infty \}
\end{equation*}%
the Hilbert space with inner product $\langle f_{1},f_{2}\rangle =\sum_{x\in
G}f_{1}(x)\overline{f_{2}(x)}$. Let $B(l^{2}(G))$ be the set of all bounded
linear operators on the Hilbert space $l^{2}(G).$ By definition, the group
von Neumann algebra $\mathcal{N}G$ is the completion of the complex group
ring $\mathbb{C}[G]$ in $B(l^{2}(G))$ with respect to the weak operator
topology. There is a continuous, additive von Neumann dimension that assigns
to every right $\mathcal{N}G$-module $M$ a value $\mathrm{dim}_{\mathcal{N}%
G}(M)\in \lbrack 0,\infty ]$ (see Definition 6.20 of \cite{lu}). For a group
$G,$ let $EG$ be the universal covering space of its classifying space $%
\mathrm{B}G.$ Denote by $C_{\ast }^{\text{sing}}(EG)$ the singular chain
complex of $EG$ with the induced $\mathbb{Z}G$-structure. The $L^{2}$%
-homology is the singular homology $H_{i}^{G}(EG;\mathcal{N}G)$ with
coefficients $\mathcal{N}G$, \textsl{i.e.} the homology of the $\mathcal{N}G$%
-chain complex $\mathcal{N}G\bigotimes_{\mathbb{Z}G}C_{\ast }^{\text{sing}%
}(EG).$ The $i$-th $L^{2}$-Betti number of $G$ is defined by%
\begin{equation*}
b_{i}^{(2)}(G):=\mathrm{dim}_{\mathcal{N}G}(H_{i}^{G}(EG;\mathcal{N}G))\in
\lbrack 0,\infty ].
\end{equation*}%
The $L^{2}$-homology and $L^{2}$-Betti numbers are important invariants of
spaces and groups. They have many applications to geometry and $K$-theory.
For more details, see the book \cite{lu}.

It has been proved that the $L^{2}$-Betti numbers are (almost) zero for
several class of groups, including amenable groups, Thompson's group (cf.
\cite{lu}, Theorem 7.20), Baumslag-Solitar group (cf. \cite{dl,bf}), mapping
class group of a closed surface with genus $g\geq 2$ except $b_{3g-3}^{(2)}$
(cf. \cite{ki}, Corollary D.15) and so on (for more information, see \cite%
{lu}, Chapter 7). Let $R$ be an associative ring with identity and $n\geq 2$
be an integer. The general linear group $\mathrm{GL}_{n}(R)$ is the group of
all $n\times n$ invertible matrices with entries in $R$. For an element $%
r\in R$ and any integers $i,j$ such that $1\leq i\neq j\leq n,$ denote by $%
e_{ij}(r)$ the elementary $n\times n$ matrix with $1s$ in the diagonal
positions and $r$ in the $(i,j)$-th position and zeros elsewhere. The group $%
E_{n}(R)$ is generated by all such $e_{ij}(r),$\textsl{\ i.e. }%
\begin{equation*}
E_{n}(R)=\langle e_{ij}(r)|1\leq i\neq j\leq n,r\in R\rangle .
\end{equation*}%
When $R$ is commutative, we define the special linear group $\mathrm{SL}%
_{n}(R)$ as the subgroup of $\mathrm{GL}_{n}(R)$ consisting of matrices with
determinants $1$. For example in the case $R=\mathbb{Z},$ the integers, we
have that $\mathrm{SL}_{n}(\mathbb{Z})=E_{n}(R).$ The groups $\mathrm{GL}%
_{n}(R)$ and $E_{n}(R)$ are important in algebraic $K$-theory.

In this article, we prove the vanishing of lower $L^{2}$-Betti numbers for
matrix groups over a large class of rings, including all infinite
commutative rings. For this, we introduce the notion of $n$-rigid rings (for
details, see Definition \ref{de}). Examples of $n$-rigid (for any $n\geq 1$)
rings contain the following (cf. Section \ref{sec}):

\begin{itemize}
\item infinite integral domains;

\item $\mathbb{Z}$-torsionfree infinite noetherian ring (may be
non-commutative);

\item infinite commutative noetherian rings (moreover, any infinite
commutative ring is $2$-rigid);

\item Finite-dimensional algebras over $n$-rigid rings.
\end{itemize}

We prove the following results.

\begin{theorem}
\label{main}Suppose $n\geq 2.$ Let $R$ be an infinite $(n-1)$-rigid ring and
$E_{n}(R)$ the group generated by elementary matrices$.$ For each $i\in
\{0,\cdots ,n-2\},$ the $L^{2}$-Betti number $b_{i}^{(2)}(E_{n}(R))=0.$
\end{theorem}

Since $b_{1}^{(2)}(E_{2}(\mathbb{Z}))\neq 0,$ the above result does not hold
for $i=n-1$ in general.

\begin{corollary}
\label{cor}\bigskip Let $R$ be any infinite commutative ring and $n\geq 2$.
For each $i\in \{0,\cdots ,n-2\},$ the $L^{2}$-Betti number
\begin{equation*}
b_{i}^{(2)}(\mathrm{GL}_{n}(R))=b_{i}^{(2)}(\mathrm{SL}%
_{n}(R))=b_{i}^{(2)}(E_{n}(R))=0.
\end{equation*}
\end{corollary}

Let $\mathrm{SL}_{n}(R)$ be a lattice in a semisimple Lie group, \textsl{e.g.%
} when $R=\mathbb{Z}$ or a subring of algebraic integers. It follows from
results of Borel, which rely on global analysis on the associated symmetric
space, that the $L^{2}$-Betti numbers of $\mathrm{SL}_{n}(R)$ vanish except
possibly in the middle dimension of the symmetric space (cf. \cite{Bo,O}).
In particular, all the $L^{2}$-Betti numbers of $\mathrm{SL}_{n}(\mathbb{Z})$
$(n\geq 3)$ are zero (cf. \cite{Ec}, Example 2.5). For any infinite integral
domain $R$ and any $i\in \{0,\cdots ,n-2\},$ Bader-Furman-Sauer \cite{bf}
proves that the $L^{2}$-Betti number $b_{i}^{(2)}(\mathrm{SL}_{n}(R))=0.$ M.
Ershov and A. Jaikin-Zapirain \cite{ej} prove that the noncommutative
universal lattices $E_{n}(\mathbb{Z}\langle x_{1},\cdots ,x_{k}\rangle )$
(and therefore $E_{n}(R)$ for any finitely generated associative ring $R$)
has Kazhdan's property (\textrm{T}) for $n\geq 3.$ This implies that for any
finitely generated associative ring $R,$ the first $L^{2}$-Betti number of $%
E_{n}(R)$ vanishes (cf. \cite{bv}).

We consider more matrix groups as follows. Let $R$ be a commutative ring
with identity. The symplectic group is defined as
\begin{equation*}
\mathrm{Sp}_{2n}(R)=\{A\in \mathrm{GL}_{2n}(R)|\,\ A^{T}\varphi
_{n}A=\varphi _{n}\},
\end{equation*}%
where $A^{T}$ is the transpose of $A$ and
\begin{equation*}
\varphi _{n}=\left(
\begin{array}{cc}
0 & I_{n} \\
-I_{n} & 0%
\end{array}%
\right) .
\end{equation*}%
Similarly, the split orthogonal group is defined as
\begin{equation*}
\mathrm{O}(n,n)(R)=\{A\in \mathrm{GL}_{2n}(R)|\,\ A^{T}\psi _{n}A=\psi _{n}\}
\end{equation*}%
where%
\begin{equation*}
\psi _{n}=\left(
\begin{array}{cc}
0 & I_{n} \\
I_{n} & 0%
\end{array}%
\right) .
\end{equation*}

For symplectic and orthogonal groups, we obtain the following.

\begin{theorem}
\label{main2}Let $R$ be an infinite principal ideal domain (PID) and $%
\mathrm{Sp}_{2n}(R)$ the symplectic group with its elementary subgroup $%
\mathrm{ESp}_{2n}(R)$ (resp. $\mathrm{O}(n,n)(R)$ the orthogonal group and
its elementary subgroup $\mathrm{EO}(n,n)(R)$)$.$ We have the following.

\begin{enumerate}
\item[(i)] For each $i=0,\cdots ,n-2$ $(n\geq 2),$ the $L^{2}$-Betti number
\begin{equation*}
b_{i}^{(2)}(\mathrm{Sp}_{2n}(R))=b_{i}^{(2)}(\mathrm{ESp}_{2n}(R))=0.
\end{equation*}

\item[(ii)] For each $i=0,\cdots ,n-2$ $(n\geq 2),$ the $L^{2}$-Betti number
\begin{equation*}
b_{i}^{(2)}(\mathrm{O}(n,n)(R))=b_{i}^{(2)}(\mathrm{EO}(n,n)(R))=0.
\end{equation*}
\end{enumerate}
\end{theorem}

\bigskip

The proofs of Theorem \ref{main} and Theorem \ref{main2} are based on a
study of the notion of weak normality of particular subgroups in matrix
groups, introduced in \cite{bf} and \cite{pt}. We present another
application of the weak normality of subgroups in matrix groups as follows.

Acylindrically hyperbolic groups are defined by Dahmani--Guirardel--Osin
\cite{dgo} and Osin \cite{Os}. Let $G$ be a group. An isometric $G$-action
on a metric space $S$ is said to be acylindrical if for every $\varepsilon
>0 $, there exist $R,N>0$ such that for every two points $x,y\in S$ with $%
d(x,y)\geq R$, there are at most $N$ elements $g\in G$ which satisfy $%
d(x,gx)\leq \varepsilon $ and $d(y,gy)\leq \varepsilon $. A $G$-action by
isometries on a hyperbolic geodesic space $S$ is said to be elementary if
the limit set of $G$ on the Gromov boundary $\partial S$ contains at most 2
points. A group $G$ is called \emph{acylindrically hyperbolic} if $G$ admits
a non-elementary acylindrical action by isometries on a (Gromov-$\delta $)
hyperbolic geodesic space. The class of acylindrically hyperbolic groups
includes non-elementary hyperbolic and relatively hyperbolic groups, mapping
class groups of closed surface $\Sigma _{g}$ of genus $g\geq 1,$ outer
automorphism group $\mathrm{Out}(F_{n})$ $(n\geq 2)$ of free groups,
directly indecomposable right angled Artin groups, 1-relator groups with at
least 3 generators, most 3-manifold groups, and many other examples.

Although there are many analogies among matrix groups, mapping class groups
and outer automorphism groups of free groups, we prove that they are
different on acylindrical hyperbolicity, as follows.

\begin{theorem}
\label{main3}Suppose that $n$ is an integer.

\begin{enumerate}
\item[(i)] Let $R$ be a $2$-rigid (\textsl{eg.} commutative) ring. The group
$E_{n}(R)$ $(n\geq 3)$ is not acylindrically hyperbolic.

\item[(ii)] Let $R$ be a commutative ring. The group $G$ is not
acylindrically hyperbolic, if $G=\mathrm{GL}_{n}(R)$ $(n\geq 3)$ the general
linear group, $\mathrm{SL}_{n}(R)$ $(n\geq 3)$ the special linear group, $%
\mathrm{Sp}_{2n}(R)$ $(n\geq 2)$ the symplectic group, $\mathrm{ESp}_{2n}(R)$
$(n\geq 2)$ the elementary symplectic group, $\mathrm{O}(n,n)(R)$ $(n\geq 4)$
the orthogonal group, or $\mathrm{EO}(n,n)(R)$ $(n\geq 4)$ the elementary
orthogonal group.
\end{enumerate}
\end{theorem}

When $R$ is commutative, the failure of acylindrical hyperbolicity of the
elementary groups $E_{n}(R),$ $\mathrm{ESp}_{2n}(R)$ and $\mathrm{EO}%
(n,n)(R) $ is already known to Mimura \cite{m} by studying property TT for
weakly mixing representations. But our approach is different and Theorem \ref%
{main3} is more general, even for elementary subgroups. Explicitly, for
noncommutative rings we have the following.

\begin{corollary}
\label{cor3}Let $R$ be a noncommutative $\mathbb{Z}$-torsionfree infinite
noetherian ring, integral group ring over a polycyclic-by-finite group or
their finite-dimensional algebra. For each nonnegative integer $i\leq n-2,$
we have
\begin{equation*}
b_{i}^{(2)}(E_{n}(R))=0.
\end{equation*}%
Furthermore, the group $E_{n}(R)$ $(n\geq 3)$ is not acylindrically
hyperbolic.
\end{corollary}

\section{$s$-normality}

\bigskip Recall from \cite{bf} that the $n$-step $s$-normality is defined as
follows.

\begin{definition}
Let $n\geq 1$ be an integer. A subgroup $H$ of a group $G$ is called $n$%
-step $s$-normal if for any $(n+1)$-tuple $\omega =(g_{0},g_{1},\cdots
,g_{n})\in G^{n+1},$ the intersection
\begin{equation*}
H^{\omega }:=\cap _{i=0}^{n}g_{i}Hg_{i}^{-1}
\end{equation*}%
is infinite. A $1$-step $s$-normal group is simply called $s$-normal.
\end{definition}

The following result is proved by Bader, Furman and Sauer (cf. \cite{bf},
Theorem 1.3).

\begin{lemma}
\label{lem1}Let $H$ be a subgroup of $G$. Assume that
\begin{equation*}
b_{i}^{(2)}(H^{\omega })=0
\end{equation*}%
for all integers $i,k\geq 0$ with $i+k\leq n$ and every $\omega \in G^{k+1}.$
In particular, $H$ is an $n$-step $s$-normal subgroup of $G.$ Then
\begin{equation*}
b_{i}^{(2)}(G)=0
\end{equation*}%
for every $i\in \{0,\ldots ,n\}.$
\end{lemma}

\bigskip The following result is important for our later arguments (cf. \cite%
{lu}, Theorem 7.2, (1-2), p.294).

\begin{lemma}
\label{lem2}Let $n$ be any non-negative integers. Then

\begin{enumerate}
\item[(i)] For any infinite amenable group $G$, the $L^{2}$-Betti numbers $%
b_{n}^{(2)}(G)=0.$

\item[(ii)] Let $H$ be a normal subgroup of a group $G$ with vanishing $%
b_{i}^{(2)}(H)=0$ for each $i\in \{0,1,\ldots ,n\}.$ Then for each $i\in
\{0,1,\ldots ,n\},$ we have $b_{i}^{(2)}(G)=0.$
\end{enumerate}
\end{lemma}

We will also need the following fact (cf. \cite{Os}, Corollary 1.5,
Corollary 7.3).

\begin{lemma}
\label{lem3}The class of acylindrically hyperbolic groups is closed under
taking $s$-normal subgroups. Furthermore, the center of an acylindrically
hyperbolic group is finite.
\end{lemma}

\section{Rigidity of rings\label{sec}}

We introduce the notion of $n$-rigidity of rings. For a ring, all $R$%
-modules are right modules and homomorphisms are right $R$-module
homomorphisms.

\begin{definition}
\label{de}For a positive integer $n,$ an infinite ring $R$ is called $n$%
\emph{-rigid} if every $R$-homomorphism $R^{n}\rightarrow R^{n-1}$ of the
free modules has an infinite kernel.
\end{definition}

A related concept is the strong rank condition: a ring $R$ satisfies the
strong rank condition if there is no injection $R^{n}\rightarrow R^{n-1}$
for any $n$ (see Lam (\cite{la}, p.12). Clearly, $n$-rigidity for any $n$
implies the strong rank condition for a ring. Fixing the standard basis of
both $R^{n}$ and $R^{n-1}$, the kernel of an $R$-homomorphism $\phi
:R^{n}\rightarrow R^{n-1}$ corresponds to a system of $n-1$ linear equations
with $n$ unknowns over $R$:
\begin{equation*}
S:\sum_{1\leq i\leq n}a_{ij}x_{i}=0,1\leq j\leq n-1,
\end{equation*}%
with $a_{ij}\in R,1\leq i\leq n,1\leq j\leq n-1.$ Therefore, the strong rank
condition asserts that the system $S$ has non-trivial solutions over $R$,
while the $n$-rigidity property requires that $S$ has infinitely many
solutions.

Many rings are $n$-rigid. For example, infinite integral rings are $n$-rigid
for any $n$ by considering the dimensions over quotient fields. Moreover,
let $A$ be a ring satisfying the strong rank condition (eg. noetherian ring,
cf. Theorem 3.15 of \cite{la}.). Suppose that $A$ is a torsion-free $\mathbb{%
Z}$-module, where $\mathbb{Z}$ acts on $A$ via $\mathbb{Z}\cdot 1_{A}$.
Since the kernel $A^{n}\rightarrow A^{n-1}$ is a nontrivial $\mathbb{Z}$%
-module, the ring $A$ is $n$-rigid for any $n.$

We present several basic facts on $n$-rigid rings as follows.

\begin{lemma}
\label{subn}$n$-rigid implies $(n-1)$-rigid.
\end{lemma}

\begin{proof}
For any $R$-homomorphism $f:$ $R^{n-1}\rightarrow R^{n-2},$ we could add a
copy of $R$ as direct summand to get a map $f\bigoplus id:R^{n-1}\bigoplus
R\rightarrow R^{n-2}\bigoplus R.$ The two maps have the same kernel.
\end{proof}

\begin{lemma}
\label{exte}Let $R$ be an $n$-rigid ring for any $n\geq 1.$ Suppose that an
associative ring $A$ is a finite-dimensional $R$-algebra (i.e. $A$ is a free
$R$-module of finite rank with compatible multiplications in $A$ and $R$).
Then $A$ is $n$-rigid for any $n\geq 1.$
\end{lemma}

\begin{proof}
Let $f:A^{n}\rightarrow A^{n-1}$ be an $A$-homomorphism. Viewing $A$ as a
finite-dimensional $R$-module, we see that $f$ is also an $R$-homomorphism.
Embed the target $A^{n-1}$ into $R^{n\cdot \mathrm{rank}_{R}(A)-1}.$ The
kernel $\ker f$ is infinite by the assumption that $R\quad $is $n\cdot
\mathrm{rank}_{R}(A)$-rigid.
\end{proof}

\begin{proposition}
\label{prop}Let $R$ be an $n$-rigid ring and $u_{1},u_{2},\cdots ,u_{n-1}\in
R^{m}$ $(m\geq n)$ be arbitrary $n-1$ elements. Then the set%
\begin{equation*}
\{\phi \in \mathrm{Hom}_{R}(R^{m},R)\mid \phi (u_{i})=0,i=1,2,\cdots ,n-1\}
\end{equation*}%
is infinite.\bigskip
\end{proposition}

\begin{proof}
When $m=n,$ we define an $R$-homomorphism%
\begin{eqnarray*}
\mathrm{Hom}_{R}(R^{n},R) &\rightarrow &R^{n-1} \\
f &\longmapsto &(f(u_{1}),f(u_{2}),\cdots ,f(u_{n-1})).
\end{eqnarray*}%
Since $\mathrm{Hom}_{R}(R^{n},R)$ is isomorphic to $R^{n}$, such an $R$%
-homomorphism has an infinite kernel. When $m>n,$ we may project $R^{m}$ to
its last $n$-components and apply a similar proof.
\end{proof}

\begin{lemma}
\label{comm}An infinite commutative ring $R$ is $2$-rigid.
\end{lemma}

\begin{proof}
Let $f:R^{2}\rightarrow R$ be any $R$-homomorphism. Denote by
\begin{equation*}
I=\langle xR+yR\mid (x,y)\in \ker f\rangle \unlhd R.
\end{equation*}%
Suppose that $\ker f$ is finite. When $(x,y)\in \ker f,$ the set $xR$ and $%
yR $ are also finite. Thus $I$ is finite. Denote by $a=f((1,0))$ and $%
b=f((0,1)).$ Note that $(-b,a)\in \ker f.$ For any $(x,y)\in R^{2},$ we have
$ax+by\in I.$ Since the set of right cosets $R/I$ is infinite, we may choose
$(x,x)$ and $(y,y)$ with $x,y$ from distinct cosets such that
\begin{equation*}
ax+bx=ay+by.
\end{equation*}%
However, $(x-y,x-y)\in \ker f$ and thus $x-y\in I.$ This is a contradiction.
\end{proof}

To state our result in the most general form, we introduce the following
notion.

\begin{definition}
A ring $R$ is called \emph{size-balanced} if any finite right ideal of $R$
generates a finite two sided ideal of $R$.
\end{definition}

It is immediate that any commutative ring is size-balanced.

\begin{proposition}
\label{prop-asin} A size-balanced infinite noetherian ring is $n$-rigid for
any $n.$
\end{proposition}

\begin{proof}
Let $f:R^{n}\rightarrow R^{n-1}$ be any $R$-homomorphism. Let $%
A=(a_{ij})_{(n-1)\times n}$ be the matrix representation of $f$ with respect
to the standard basis. Denote by
\begin{equation*}
I^{\prime }=\langle x_{1}R+x_{2}R+\cdots +x_{n}R\mid (x_{1},x_{2},\cdots
,x_{n})\in \ker f\rangle \unlhd R.
\end{equation*}%
First we notice that $I^{\prime }$ is non-trivial by the strong rank
condition of noetherian rings (cf. Theorem 3.15 of \cite{la}). Suppose that $%
\ker f$ is finite. For any
\begin{equation*}
(x_{1},x_{2},\cdots ,x_{n})\in \ker f
\end{equation*}
and $r\in R$, each $(x_{1}r,x_{2}r,\cdots ,x_{n}r)\in \ker f$. As $\ker f$
is finite, each right ideal $x_{i}R$ is finite; and hence so is $I^{\prime }$%
. Let $I$ be the two sided ideal generated by the finite right ideal $%
I^{\prime }$. It is finite as $R$ is assumed to be size-balanced. Therefore,
the quotient ring $R/I$ is infinite and noetherian.

Let $\bar{f}:R/I\rightarrow R/I$ be the $R/I$-homomorphism induced by the
matrix $\bar{A}=(\bar{a}_{ij}),$ where $\bar{a}_{ij}$ is the image of $%
a_{ij}.$ If $(\bar{x}_{1},\bar{x}_{2},\cdots ,\bar{x}_{n})\in \ker \bar{f}$
and $x_{i}$ is any pre-image of $\bar{x}_{i}$, we have%
\begin{equation*}
A(x_{1},x_{2},\cdots ,x_{n})^{T}\in I^{n-1}.
\end{equation*}%
As $I$ is finite, so is $I^{n-1}$. If $\ker \bar{f}$ is infinite, there are
two distinct elements in $\ker \bar{f}$ with pre-image $(x_{1},x_{2},\cdots
,x_{n})$ and $(y_{1},y_{2},\cdots ,y_{n})$ in $R^{n}$ such that%
\begin{equation*}
A(x_{1},x_{2},\cdots ,x_{n})^{T}=A(y_{1},y_{2},\cdots ,y_{n})^{T}\in I^{n-1}.
\end{equation*}%
However, this implies that
\begin{equation*}
(x_{1},x_{2},\cdots ,x_{n})-(y_{1},y_{2},\cdots ,y_{n})\in \ker f.
\end{equation*}%
We have a contradiction as $(x_{1},x_{2},\cdots ,x_{n})$ and $%
(y_{1},y_{2},\cdots ,y_{n})$ are distinct in $(R/I)^{n}$. Therefore, $\ker
\bar{f}$ is finite. Moreover, $\ker \bar{f}$ is non-trivial by the strong
rank condition of noetherian rings. Let $I_{1}^{\prime }$ be the pre-image
of the right ideal generated by components of elements in $\ker \bar{f}$ in $%
R$, which is a finite right ideal by a similar argument as above. It
generates a finite two sided ideal $I_{1}$ of $R,$ and it properly contains $%
I.$

Repeating the argument, we get an infinite ascending sequence
\begin{equation*}
I\lneqq I_{1}\lneqq I_{2}\lneqq \cdots
\end{equation*}
of finite ideals of $R$. This is a contradiction to the assumption that $R$
is noetherian.
\end{proof}

\begin{corollary}
\label{coro-acrr} Any commutative ring $R$ containing an infinite noetherian
subring is $n$-rigid for each $n$.
\end{corollary}

\begin{proof}
Let $R_{0}$ be an infinite noetherian subring of $R$. Let
\begin{equation*}
S:\sum_{1\leq i\leq n}a_{ij}x_{i}=0,1\leq j\leq m
\end{equation*}%
be a system of linear equations with $a_{ij}\in R.$ Form the infinite
commutative subring $R^{\prime }$ of $R$: $R^{\prime }=R_{0}[a_{ij},1\leq
i\leq n,1\leq j\leq m].$ By the Hilbert basis theorem, $R^{\prime }$ is
infinite noetherian. Proposition \ref{prop-asin} asserts that the system $S$
has infinitely many solutions in $R^{\prime }$, and hence in $R$.
\end{proof}

\begin{example}
\label{eg}Let $G$ be a polycyclic-by-finite group and $R=\mathbb{Z}[G]$ be
its integral group ring. It is known that $R$ is infinite noetherian (see
\cite{Jat74}). Moreover, $R$ is size-balanced by the trivial reason that
there are no non-trivial finite right ideals. According to Proposition \ref%
{prop-asin}, the ring $R$ is $n$-rigid for any $n$.
\end{example}

\begin{example}
Let $F$ be a nonabelian free group and $\mathbb{Z[}F]$ the group ring. Since
$\mathbb{Z[}F]$ does not satisfy the strong rank condition (cf. \cite{la},
Exercise 29, p.21.), the ring $\mathbb{Z[}F]$ is not $n$-rigid for any $%
n\geq 2.$
\end{example}

\section{Proofs}

Let
\begin{equation*}
Q=\Big\{%
\begin{pmatrix}
1 & x \\
0 & A%
\end{pmatrix}%
\mid x\in R^{n-1},A\in \mathrm{GL}_{n-1}(R),%
\begin{pmatrix}
1 & 0 \\
0 & A%
\end{pmatrix}%
\in E_{n}(R)\}.
\end{equation*}%
It is straightforward that $Q$ contains the normal subgroup
\begin{equation*}
S=\Big\{%
\begin{pmatrix}
1 & x \\
0 & I_{n-1}%
\end{pmatrix}%
\mid x\in R^{n-1}\},
\end{equation*}%
an abelian group. Therefore, all the $L^{2}$-Betti numbers of $S$ and $Q$
are zero when the ring $R$ is infinite.

\begin{lemma}
\label{ke}Let $k<n$ $(n\geq 3)$ be two positive integers. Suppose that $R$
is an infinite $k$-rigid ring. The subgroup $Q$ is $(k-1)$-step $s$-normal
in $E_{n}(R).$ In particular, $Q$ is $s$-normal if $R$ is infinite $2$-rigid.
\end{lemma}

\begin{proof}
Without loss of generality, we assume $k=n-1.$ Let $g_{1},g_{2},\cdots
,g_{n-2}\ $be any $n-2$ elements in $E_{n}(R).$ We will show that the
intersection $Q\cap _{i=1}^{n-2}g_{i}Qg_{i}^{-1}$ is infinite, which implies
the $(k-1)$-step $s$-normality of $H$. Let $\{e_{i}\}_{i=1}^{n}$ be the
standard basis of $R^{n}.$ Denote by $U=R^{n-1}$ the $R$-submodule spanned
by $\{e_{i}\}_{i=2}^{n}$ and $p:R^{n}\rightarrow U$ the natural projection.

For each $g_{i},i=1,2,\cdots ,n-2,$ suppose that
\begin{equation*}
g_{i}e_{1}=x_{i}e_{1}+u_{i}
\end{equation*}%
for $x_{i}\in R$ and $u_{i}\in U.$ Denote
\begin{equation*}
\Phi =\{\phi \in \mathrm{Hom}_{R}(U,R)\mid \phi (u_{i})=0,i=1,2,\cdots
,n-2\}.
\end{equation*}

For any $\phi \in \Phi ,$ let $T_{\phi }:R^{n}\rightarrow R^{n}$ defined by $%
T_{\phi }(v)=v+\phi \circ p(v)e_{1}.$ It is obvious that
\begin{equation*}
g_{i}^{-1}T_{\phi }g_{i}(e_{1})=e_{1}
\end{equation*}%
for each $i=1,2,\cdots ,n-2.$ Note that $Q$ is the stabilizer of $e_{1}$.
This shows that for each $\phi \in \Phi ,$ the transformation $T_{\phi }$
lies in $Q\cap _{i=1}^{n-2}g_{i}Qg_{i}^{-1}.$ Denote by $T$ the subgroup
\begin{equation*}
T=\{T_{\phi }\mid \phi \in \Phi \}.
\end{equation*}%
By Proposition \ref{prop}, $\Phi $ is infinite and thus $T$ is infinite. The
proof is finished.
\end{proof}

\begin{lemma}
\label{new}In the proof of Lemma \ref{ke}, the subgroup $T$ is normal in $%
Q\cap _{i=1}^{n-2}g_{i}Qg_{i}^{-1}.$
\end{lemma}

\begin{proof}
For any $\phi ,$ write $e_{\phi }=(\phi (e_{2}),\cdots ,\phi (e_{n})).$ With
respect to the standard basis, the representation matrix of the
transformation $T_{\phi }$ is $%
\begin{pmatrix}
1 & e_{\phi } \\
0 & I_{n-1}%
\end{pmatrix}%
.$ For any $%
\begin{pmatrix}
1 & x \\
0 & A%
\end{pmatrix}%
\in Q\cap _{i=1}^{n-2}g_{i}Qg_{i}^{-1},$ the conjugate
\begin{equation*}
\begin{pmatrix}
1 & x \\
0 & A%
\end{pmatrix}%
^{-1}%
\begin{pmatrix}
1 & e_{\phi } \\
0 & I_{n-1}%
\end{pmatrix}%
\begin{pmatrix}
1 & x \\
0 & A%
\end{pmatrix}%
=%
\begin{pmatrix}
1 & e_{\phi }A \\
0 & I_{n-1}%
\end{pmatrix}%
.
\end{equation*}%
Define $\psi :U=R^{n-1}\rightarrow R$ by $\psi (x)=e_{\phi }Ax.$ For each $%
i=1,\cdots ,n-2,$ we have that $%
\begin{pmatrix}
1 & x \\
0 & A%
\end{pmatrix}%
=g_{i}q_{i}g_{i}^{-1}$ for some $q_{i}\in Q.$ Therefore,
\begin{equation*}
\begin{pmatrix}
1 & x \\
0 & A%
\end{pmatrix}%
g_{i}e_{1}=g_{i}q_{i}e_{1}
\end{equation*}%
and $Au_{i}=u_{i}.$ This implies that $\psi (u_{i})=e_{\phi }u_{i}=0$ for
each $i$ and thus $\psi \in \Phi .$ Therefore, the conjugate $%
\begin{pmatrix}
1 & e_{\phi }A \\
0 & I_{n-1}%
\end{pmatrix}%
$ lies in $T,$ which proves that $T$ is normal.
\end{proof}

\bigskip

\begin{proof}[Proof of Theorem \protect\ref{main}]
By Lemma \ref{new}, any intersection $Q\cap _{i=1}^{n-2}g_{i}Qg_{i}^{-1}$
contains an infinite normal amenable subgroup $T.$ Therefore, all the $L^{2}$%
-Betti numbers of any intersection $Q\cap _{i=1}^{k}g_{i}Qg_{i}^{-1}$ are
vanishing for $k\leq n-2$ considering Lemma \ref{subn}. We have that $%
b_{i}(E_{n}(R))=0$ for any $0\leq i\leq n-2$ by Lemma \ref{lem1}.
\end{proof}

\bigskip

\begin{proof}[Proof of Corollary \protect\ref{cor}]
When $n=2,$ it is clear that both $\mathrm{GL}_{n}(R)$ and $\mathrm{SL}%
_{n}(R)$ are infinite, since $E_{2}(R)$ is an infinite subgroup. Thus $%
b_{0}^{(2)}(\mathrm{GL}_{2}(R))=b_{0}^{(2)}(\mathrm{SL}_{2}(R))=0.$ We have
already proved that $b_{i}^{(2)}(E_{n}(S))=0$ for infinite commutative
noetherian ring $S$ and $0\leq i\leq n-2,$ since the ring $S$ would be $k$%
-rigid for any integer $k$ by Proposition \ref{prop-asin}. If $S$ is a
finite subring of $R$, the group $E_{n}(S)$ is also finite. Therefore, we
still have $b_{i}^{(2)}(E_{n}(S))=0$ for $1\leq i\leq n-2.$ Note that every
commutative ring $R$ is the directed colimit of its subrings $S$ that are
finitely generated as $\mathbb{Z}$-algebras (noetherian rings by Hilbert
basis theorem). Since the group $E_{n}(R)$ is the union of the directed
system of subgroups $E_{n}(S),$ we get that
\begin{equation*}
b_{i}^{(2)}(E_{n}(R))=0
\end{equation*}%
for $0\leq i\leq n-2$ (cf. \cite{lu}, Theorem 7.2 (3) and its proof)$.$ When
$R$ is commutative and $n\geq 3$, a result of Suslin says that the group $%
E_{n}(R)$ is a normal subgroup of $\mathrm{GL}_{n}(R)$ and $\mathrm{SL}%
_{n}(R)$ (cf. \cite{Su}). Lemma \ref{lem2} implies that $b_{i}^{(2)}(\mathrm{%
GL}_{n}(R))=b_{i}^{(2)}(\mathrm{SL}_{n}(R))=0$ for each $i\in \{0,\ldots
,n-2\}.$
\end{proof}

\bigskip

We follow \cite{bak} to define the elementary subgroups of symplectic groups
and orthogonal groups. Let $E_{ij}$ denote the $n\times n$ matrix with $1$
in the $(i,j)$-th position and zeros elsewhere. Then for $i\neq j,$ the
matrix $e_{ij}(a)=I_{n}+aE_{ij}$ is an elementary matrix, where $I_{n}$ is
the identity matrix of size $n$. With $n$ fixed, for any integer $1\leq
k\leq 2n,$ set $\sigma k=k+n$ if $k\leq n$ and $\sigma k=k-n$ if $k>n$. For $%
a\in R$ and $1\leq i\neq j\leq 2n,$ we define the elementary unitary
matrices $\rho _{i,\sigma i}(a)$ and $\rho _{ij}(a)$ with $j\neq \sigma i$
as follows:

\begin{itemize}
\item $\rho _{i,\sigma i}(a)=I_{2n}+aE_{i,\sigma i}$ with $a\in R$;

\item Fix $\varepsilon =\pm 1.$ We define $\rho _{ij}(a)=\rho _{\sigma
j,\sigma i}(-a^{\prime })=I_{2n}+aE_{ij}-a^{\prime }E_{\sigma j,\sigma i}$
with $a^{\prime }=a$ when $i,j\leq n$; $a^{\prime }={\varepsilon a}$ when $%
i\leq n<j$; $a^{\prime }=a\varepsilon $ when $j\leq n<i$; and $a^{\prime }=a$
when $n+1\leq i,j$.
\end{itemize}

When $\varepsilon =-1,$ we have the elementary symplectic group
\begin{equation*}
\mathrm{ESp}_{2n}(R)=\langle \rho _{i,\sigma i}(a),\rho _{ij}(a)\mid a\in
R,i\neq j,i\neq \sigma j\rangle .
\end{equation*}%
When $\varepsilon =1,$ we have the elementary orthogonal group%
\begin{equation*}
\mathrm{EO}(n,n)(R)=\langle \rho _{ij}(a)\mid a\in R,i\neq j,i\neq \sigma
j\rangle .
\end{equation*}%
Note that for the orthogonal group, each matrix $\rho _{i,\sigma i}(a)$ is
not in $\mathrm{EO}(n,n)(R).$

There is an obvious embedding
\begin{equation*}
\mathrm{Sp}_{2n}(R)\rightarrow \mathrm{Sp}_{2n+2}(R),
\end{equation*}%
\begin{equation*}
\left(
\begin{array}{ll}
\alpha  & \beta  \\
\gamma  & \delta
\end{array}%
\right) \longmapsto \left(
\begin{array}{llll}
1 & 0 & 0 & 0 \\
0 & \alpha  & 0 & \beta  \\
0 & 0 & 1 & 0 \\
0 & \gamma  & 0 & \delta
\end{array}%
\right) .\
\end{equation*}%
Denote the image of $A\in \mathrm{Sp}_{2n}(R)$ by $I\bigoplus A\in \mathrm{Sp%
}_{2n+2}(R).$ Let
\begin{equation*}
Q_{1}=\langle (I\oplus A)\cdot \Pi _{i=1}^{2n}\rho _{1i}(a_{i})\mid
a_{i}\in R,A\in \mathrm{Sp}_{2n-2}(R),I\oplus A\in \mathrm{ESp}%
_{2n}(R)\rangle ;
\end{equation*}%
and
\begin{equation*}
S_{1}=\langle \Pi _{i=1}^{2n}\rho _{1i}(a_{i})\mid a_{i}\in R\rangle .
\end{equation*}%
Similarly, we can define \bigskip
\begin{equation*}
\begin{split}
Q_{2}=\langle (I\oplus A)\cdot \Pi _{i=1,i\neq n+1}^{2n}\rho
_{1i}(a_{i})& \mid a_{i}\in R,A\in \mathrm{O}(2n-2,2n-2)(R), \\
I\oplus A& \in \mathrm{EO}(n,n)(R)\rangle ;
\end{split}%
\end{equation*}%
and
\begin{equation*}
S_{2}=\langle \Pi _{i=1,,i\neq n+1}^{2n}\rho _{1i}(a_{i})\mid a_{i}\in
R\rangle .
\end{equation*}

Since $S_{i}$ is abelian and normal in $Q_{i},$ all the $L^{2}$-Betti
numbers of $Q_{i}$ vanish for $i=1,2$.

\bigskip

\begin{proof}[Proof of Theorem \protect\ref{main2}]
We prove the theorem by induction on $n.$ When $n=2,$ both $\mathrm{Sp}%
_{2n}(R)$ and $\mathrm{O}(n,n)(R)$ are infinite and therefore we have $$%
b_{0}^{(2)}(\mathrm{Sp}_{4}(R))=b_{0}^{(2)}(\mathrm{O}(4,4)(R))=0.$$ The
subgroup $\mathrm{ESp}_{2n}(R)$ is normal in $\mathrm{Sp}_{2n}(R)$ when $%
n\geq 2$ and the subgroup $\mathrm{EO}(n,n)(R)$ is normal in $\mathrm{O}%
(n,n)(R)$ when $n\geq 3$ (cf. \cite{ba}, Cor. 3.10). It suffices to prove
the vanishing of Betti numbers for $G=\mathrm{ESp}_{2n}(R)$ and $\mathrm{EO}%
(n,n)(R).$

We check the condition of Lemma \ref{lem1} for $Q=Q_{1}$ (resp. $Q_{2}$) as
follows. Note that
\begin{equation*}
Q=\{g\in G\mid ge_{1}=e_{1}\}.
\end{equation*}%
Let $g_{1},g_{2},\cdots ,g_{k}$ $(g_{0}=I_{2n},k\leq n-2)$ be any $k$%
-elements in $G$ and $$K=\langle g_{0}e_{1},g_{1}e_{1},\cdots
,g_{k}e_{1}\rangle $$ the submodule in $R^{2n}$ generated by all $g_{i}e_{1}$%
. Recall that the symplectic (resp. orthogonal) form $\langle -,-\rangle
:R^{2n}\times R^{2n}\rightarrow R$ is defined by $\langle x,y\rangle
=x^{T}\varphi _{n}y$ (resp. $\langle x,y\rangle =x^{T}\psi _{n}y$). Denote
\begin{equation*}
C:=\{v\in R^{2n}\mid \langle v,g_{i}e_{1}\rangle =0\text{ for each }%
i=0,\cdots ,k-1\}.
\end{equation*}%
Let $\varepsilon =-1$ for $\mathrm{ESp}_{2n}(R)$ and $1$ for $\mathrm{EO}%
(n,n)(R)$. For each $r\in R,$ set $\delta _{\varepsilon }^{r}=r$ if $%
\varepsilon =-1$ and $\delta _{\varepsilon }^{r}=0$ if $\varepsilon =1.$ For
each $u,v\in C$ with $\langle u,u\rangle =$ $\langle u,v\rangle =\langle
v,v\rangle =0$, define the the transvections in $G$ (cf. \cite{va}, p.287,
Eichler transformations in \cite{M}, p.214, p.223-224)
\begin{eqnarray*}
\tau (u,v) &:&R^{2n}\rightarrow R^{2n}\text{ by }x\mapsto x+\varepsilon
u\langle v,x\rangle -v\langle u,x\rangle , \\
\tau _{v,r} &:&R^{2n}\rightarrow R^{2n}\text{ by }x\mapsto x-\delta
_{\varepsilon }^{r}v\langle v,x\rangle .
\end{eqnarray*}%
Note that $\tau _{v,r}$ is non-identity only in $\mathrm{ESp}_{2n}(R).$ We
have
\begin{equation*}
\tau (u,v)(g_{i}e_{1})=\tau _{v,r}(g_{i}e_{1})=g_{i}e_{1}\text{ }
\end{equation*}%
for each $i.$ Therefore, the transvections $\tau (u,v),\tau _{v,r}\in \cap
_{i=0}^{k}g_{i}Qg_{i}^{-1}.$ Let
\begin{equation*}
T=\langle \tau (u,v),\tau _{v,r}\mid u,v\in C,\langle u,u\rangle =\langle
u,v\rangle =\langle v,v\rangle =0,r\in R\rangle ,
\end{equation*}%
the subgroup generated by the transvections in $G.$ For any $g\in $ $\cap
_{i=0}^{k}g_{i}Qg_{i}^{-1},$ we have $gg_{i}e_{1}=g_{i}e_{1}$ and thus
\begin{equation*}
\langle gu,g_{i}e_{1}\rangle =\langle gu,gg_{i}e_{1}\rangle =\langle
u,g_{i}e_{1}\rangle =0.
\end{equation*}%
This implies that $g\tau (u,v)g^{-1}=\tau (gu,gv)\in T$ and $g\tau
_{v,r}g^{-1}=\tau _{gv,r}\in T.$ Therefore, the subgroup $T$ is a normal
subgroup in $\cap _{i=0}^{k}g_{i}Qg_{i}^{-1}.$

When $R$ is a PID, both the submodule $K$ and the complement $C$ are free of
smaller ranks.

\begin{description}
\item[Case (i)] $K\cap C=0.$

\item Since $R^{2n}=K\bigoplus C$ (note that each $g_{i}e_{1}$ is
unimodular), the symplectic (resp. orthogonal) form on $R^{2n}$ restricts to
a non-degenerate symplectic (resp. orthogonal) form on $C.$ Let $T<G$ as
defined before. It is known that the transvections generate the elementary
subgroups (cf. \cite{M}, p.223-224) and thus $T\cong \mathrm{ESp}_{2m}(R)$
(resp. $\mathrm{EO}(m,m)(R)$) for $m=\mathrm{rank}(C)\leq n-2.$ Since $k\leq
n-2,$ we have $m\geq 4.$ By induction,
\begin{equation*}
b_{s}^{(2)}(\cap _{i=0}^{k}g_{i}Qg_{i}^{-1})=b_{s}^{(2)}(T)=0
\end{equation*}%
for $s\leq \mathrm{rank}(C)/2-2.$ When $s+k\leq n-2,$ we have that $s\leq
\mathrm{rank}(C)/2-2,$ since $\mathrm{rank}(C)\geq 2n-(k+1).$ Therefore, $%
b_{s}^{(2)}(\cap _{i=0}^{k}g_{i}Q_{1}g_{i}^{-1})=0$ and Lemma \ref{lem1}
implies that
\begin{equation*}
b_{i}^{(2)}(G)=0
\end{equation*}%
for any $i\leq n-2.$

\item[Case (ii)] $K\cap C\neq 0.$

\item For any $u,v\in K\cap C$ and any $g\in \cap _{i=0}^{k}g_{i}Qg_{i}^{-1},
$ we have that $gu=u,gv=v$ and
\begin{equation*}
g\tau (u,v)g^{-1}=\tau (gu,gv)=\tau (u,v).
\end{equation*}%
This implies that $\tau (u,v)$ lies in the center of $\cap
_{i=0}^{k}g_{i}Qg_{i}^{-1}.$ Note that when $G=\mathrm{ESp}_{2n}(R),$ the
transvection $\tau (u,u)$ is not trivial for any $u\in K\cap C$. When $G=%
\mathrm{EO}(n,n)(R)$ and $\mathrm{rank}(K\cap C)\geq 2,$ the transvection $%
\tau (u,v)$ is not trivial for any linearly independent $u,v\in K\cap C$.
Moreover, for two elements $r,s$ with $r^{2}\neq s^{2},$ we have $\tau
(ru,rv)\neq \tau (su,sv)$ when $\tau (u,v)\neq I_{2n}$ (take note that for $%
G=\mathrm{ESp}_{2n}(R)$, we can just let $u=v$ from above). The infinite PID
$R$ contains infinitely many square elements. In summary, as $K\cap C$ is a
free $R$-module, the subgroup
\begin{equation*}
T^{\prime }=\langle \tau (u,v)\mid v\in K\cap C\rangle <G
\end{equation*}%
is an infinite abelian normal subgroup of $\cap _{i=0}^{k}g_{i}Qg_{i}^{-1}.$
Therefore,%
\begin{equation*}
b_{s}^{(2)}(\cap _{i=0}^{k}g_{i}Qg_{i}^{-1})=b_{s}^{(2)}(T^{\prime })=0
\end{equation*}%
for each integer $s\geq 0$. Therefore, for any $i\leq n-2,$ we have that $%
b_{i}^{(2)}(G)=0$ by Lemma \ref{lem1}.

The remaining situation is that $G=\mathrm{EO}(n,n)(R)$ and $\mathrm{rank}%
(K\cap C)=1.$ Choose the decomposition $C=(K\cap C)\bigoplus C_{1}.$ The
orthogonal form restricts to a non-degenerate orthogonal form on $C_{1}$
(suppose that for some $x\in C_{1}$ we have $\langle x,y\rangle =0$ for any $%
y\in C_{1}$. Since $\langle x,k\rangle =0$ for any $k\in K,$ we know that $%
\langle x,y\rangle =0$ for any $y\in C.$ This implies $x\in K,$ which gives $%
x=0$). Since $k\leq n-2,$ the even number $\mathrm{rank}(C_{1})\geq 4.$ A
similar argument as case (i) finishes the proof.
\end{description}
\end{proof}

\begin{remark}
\label{nrm} Let $T$ be the normal subgroup of $\cap
_{i=0}^{k}g_{i}Qg_{i}^{-1}$ constructed in the proof of Theorem \ref{main2}.
We do not know whether the $L^{2}$-Betti numbers $b_{i}^{(2)}(T)=0$ for a
general infinite $(2n-1)$-rigid commutative ring $R$ when $i\leq n-2-k$. If
yes, Theorem \ref{main2} would hold for any general infinite commutative
ring by a similar argument as the proof of Corollary \ref{cor}.
\end{remark}

\begin{proof}[Proof of Theorem \protect\ref{main3}]
Note that when $R$ is commutative, the elementary subgroups $E_{n}(R),%
\mathrm{ESp}_{2n}(R)$ and $\mathrm{EO}(n,n)(R)$ are normal in $\mathrm{SL}%
_{n}(R),\mathrm{Sp}_{2n}(R)$ and $\mathrm{O}(n,n)(R),$ respectively (cf.
\cite{Su}, \cite{ba} Cor. 3.10). Therefore, it is enough to prove the
failure of acylindrically hyperbolicity for elementary subgroups. We prove
(i) first. If $R\ $is finite, all the groups will be finite and thus not
acylindrically hyperbolic. If $R$ is infinite, then it is $2$-rigid and the
subgroup $Q$ is $s$-normal by Lemma \ref{ke}. Suppose that $E_{n}(R)$ is
acylindrically hyperbolic. Lemma \ref{lem3} implies that both $Q$ and $S$
are acylindrically hyperbolic. However, the subgroup $S$ is infinite
abelian, which is a contradiction to the second part of Lemma \ref{lem3}.

For (ii), we may also assume that $R$ is infinite since any finite group is
not acylindrically hyperbolic. It suffices to prove that $Q_{1}$ (resp. $%
Q_{2}$) is $s$-normal in $\mathrm{ESp}_{2n}(R)$ (resp. $\mathrm{EO}(n,n)(R)$%
). (Note that $Q_{1}$ and $Q_{2}$ contain the infinite normal subgroups $%
S_{1}$ and $S_{2}$, respectively. If $G$ is acylindrically hyperbolic, the
infinite abelian subgroup $S_{1}$ or $S_{2}$ would be acylindrically
hyperbolic. This is a contradiction to the second part of Lemma \ref{lem3}.)
By definition, this is to prove that for any $g\in G,$ the intersection $%
Q\cap g^{-1}Qg$ is infinite for $Q=Q_{1}$ and $Q_{2}$. Denote by $%
ge_{1}=(x_{1},\cdots ,x_{n},y_{1},\cdots ,y_{n})^{T}.$ Let%
\begin{equation*}
t_{A}=\Pi _{1\leq i<j\leq n}\rho _{i,n+j}(a_{ij})=\left(
\begin{array}{ll}
I_{n} & A \\
0 & I_{n}%
\end{array}%
\right) \in G,
\end{equation*}%
where $A=(a_{ij})$ is an $n\times n$ matrices with entries in $R.$ Note that
$a_{ji}=a_{ij}$ or $-a_{ij}$ depending on $G=\mathrm{ESp}_{2n}(R)$ or $%
\mathrm{EO}(n,n)(R)$. Moreover, we have $\rho _{i,n+i}(a)\notin \mathrm{EO}%
(n,n)(R)$ and $\rho _{i,n+i}(a)\in \mathrm{ESp}_{2n}(R)$ for any $a\in R.$
Direct calculation shows that $t_{A}(ge_{1})-ge_{1}=((y_{1},\cdots
,y_{n})A^{T},0,\cdots ,0)^{T}.$ When $n\geq 4$ and $G=\mathrm{EO}(n,n)(R),$
the map $f:R^{\frac{n(n-1)}{2}}\rightarrow R^{n}$ defined by
\begin{equation*}
(a_{ij})_{1\leq i<j\leq n}\longmapsto A(y_{1},\cdots ,y_{n})^{T}
\end{equation*}%
has an infinite kernel $\ker f$ by $2$-rigidity of infinite commutative
rings. This implies that $\langle t_{A}\mid (a_{ij})_{1\leq i<j\leq n}\in
\ker f\rangle <Q\cap g^{-1}Qg$ is infinite. When $n\geq 2$ and $G=\mathrm{ESp%
}_{2n}(R),$ the map $f:R^{\frac{n(n+1)}{2}}\rightarrow R^{n}$ defined by $%
(a_{ij})_{1\leq i\leq j\leq n}\longmapsto A(y_{1},\cdots ,y_{n})^{T}$ has an
infinite kernel $\ker f$ and $Q\cap g^{-1}Qg$ is infinite by a similar
argument. The proof is finished.
\end{proof}

\bigskip

\begin{proof}[Proof of Corollary \protect\ref{cor3}]
\ By Proposition \ref{prop-asin}, Example \ref{eg} and Lemma \ref{exte}, all
these rings are $n$-rigid for any $n\geq 1.$ The corollary follows Theorem %
\ref{main} and Theorem \ref{main3}.
\end{proof}

\bigskip

\noindent \textbf{Acknowledgements}

The authors would like to thank the referee for pointing out a gap in a
previous version of this paper. The second author is supported by Jiangsu
Natural Science Foundation (No. BK20140402) and NSFC (No. 11501459).


\begin{thebibliography}{99}
\bibitem{bf} U. Bader, A. Furman, and R. Sauer. \textit{Weak notions of
normality and vanishing up to rank in }$L^{2}$\textit{-cohomology}.
International Mathematics Research Notices, 12 (2014), 3177-3189.

\bibitem{bak} A. Bak, $K$\textit{-theory of forms.} Ann. of Math. Stud.,
vol. 98. Princeton Univ. Press, Princeton, 1981.

\bibitem{ba} A. Bak and N. Vavilov, \textit{Normality for elementary
subgroup functors}, Math. Proc. Cambridge Phil. Soc. 118 (1995), no. No. 1,
35-47.

\bibitem{bv} M. Bekka and A. Valette, \textit{Group cohomology, harmonic
functions and the first }$L^{2}$\textit{-Betti number}. Potential Anal.,
6(1997), 313-326.

\bibitem{Bo} A. Borel, \textit{The }$L^{2}$\textit{-cohomology of negatively
curved Riemannian symmetric spaces}, Ann. Acad. Sci. Fenn. Ser. A I Math. 10
(1985), 95-105.

\bibitem{dgo} F. Dahmani, V. Guirardel, D. Osin, \textit{Hyperbolically
embedded subgroups and rotating families in groups acting on hyperbolic
spaces}, arXiv:1111.7048.

\bibitem{dl} W. Dicks and P. A. Linnell, $L^{2}$\textit{-Betti numbers of
one-relator groups}, Math. Ann. 337 (2007), no. 4, 855--874.

\bibitem{Ec} B. Eckmann. \textit{Lattices, }$l^{2}$\textit{-Betti numbers,
deficiency, and knot groups}, L'Enseignement Math. 50 (2004), 123-137.

\bibitem{ej} M. Ershov and A. Jaikin-Zapirain, \textit{Property (}$T$\textit{%
) for noncommutative universal lattices}, Invent. Math. 179 (2010), no. 2,
303--347.

\bibitem{M} A. J. Hahn \& O. T. O'Meara, \textit{The Classical Groups and
K-Theory}, Grundl. Math. Wissen. 291 Springer-Verlag (Berlin, 1989).

\bibitem{Jat74} A. Jategaonkar, \textit{Integral group rings of
polycyclic-by-finite groups}, J. of Pure and Applied Algebra, 4(1974),
337-343.

\bibitem{ki} Y. Kida, \textit{The Mapping Class Group from the Viewpoint of
Measure Equivalence Theory}, Memoirs of the American Mathematical Society
No.916, American Mathematical Soc., 2008.

\bibitem{la} T. Y. Lam, \textit{Lectures on modules and rings, Graduate
Texts in Math.}, Vol. 189, Springer-Verlag, 1999.

\bibitem{lu} W. L\"{u}ck, $L^{2}$\textit{-invariants: theory and
applications to geometry and K-theory}, Springer-Verlag, Berlin, 2002.

\bibitem{m} S. Mimura, \textit{Superrigidity from Chevalley groups into
acylindrically hyperbolic groups via quasi-cocycles}, to appear in J. Eur.
Math. Soc., arXiv:1502.03703.

\bibitem{O} M. Olbrich, $L^{2}$\textit{-invariants of locally symmetric
spaces}, Doc. Math. 7 (2002), 219-237.

\bibitem{Os} D. Osin, \textit{Acylindrically hyperbolic groups}, Trans.
Amer. Math. Soc. 368 (2016), 851-888.

\bibitem{pt} J. Peterson and A. Thom, \textit{Group cocycles and the ring of
affiliated operators}, Invent. Math. 185 (2011), no. 3, 561-592.

\bibitem{Su} A. Suslin, \textit{On the structure of the special linear group
over polynomial rings}, Math. USSR Izv: 11 (1977), 221-238, translated from
Izv. Akad. Nauk SSSR Ser. Mat. 41 (1977), 235-252.

\bibitem{va} L.N. Vaserstein, \textit{Stabilization for classical groups
over rings}, Mat. Sbornik 93:2 ( 1974), 268-295 = Math. USSR Sbornik 22,
271-303.
\end{thebibliography}
\end{document}